\documentclass[11pt]{amsart}

\usepackage{amsfonts, amsmath, amscd}
\usepackage[psamsfonts]{amssymb}

\usepackage{amssymb}

\usepackage{pb-diagram}

\usepackage[all,cmtip]{xy}

\usepackage[usenames]{color}

\newtheorem{theorem}{Theorem}[section]
\newtheorem{lemma}[theorem]{Lemma}
\newtheorem{corollary}[theorem]{Corollary}
\newtheorem{question}[theorem]{Question}

\newtheorem{proposition}[theorem]{Proposition}

\newtheorem{example}[theorem]{Example}

\numberwithin{equation}{section}

\newcommand{\CC}{C_k}
\newcommand{\NN}{\mathbb{N}}

\newcommand{\TTT}{\mathcal{T}}

\newcommand{\IR}{\mathbb{R}}

\newcommand{\Ss}{\mathbb{S}}

\newcommand{\e}{\varepsilon}
\newcommand{\cl}{\mathrm{cl}}

\renewcommand{\phi}{\varphi}

\newcommand{\U}{\mathcal U}

\newcommand{\supp}{\mathrm{supp}}

\input xy
\xyoption{all}

\title[On the Mackey problem for free locally convex spaces]{On the Mackey problem for free locally convex spaces}

\author[S.~Gabriyelyan]{Saak Gabriyelyan}
\address{Department of Mathematics, Ben-Gurion University of the
Negev, Beer-Sheva, P.O. 653, Israel}
\email{saak@math.bgu.ac.il}



\begin{document}

\begin{abstract}
We show that the free locally convex space $L(X)$ over a   Tychonoff space $X$  is a Mackey group iff $L(X)$ is a Mackey space iff $X$ is discrete.
\end{abstract}

\maketitle


\section{Introduction}


Let $(E,\tau)$ be a locally convex space (lcs for short). A locally convex vector topology $\nu$ on $E$ is called {\em compatible with $\tau$} if the spaces $(E,\tau)$ and $(E,\nu)$ have the same topological dual space. The classical  Mackey--Arens theorem states that for every lcs $(E,\tau)$ there exists the finest locally convex vector space topology $\mu$ on $E$ compatible with $\tau$. The topology $\mu$ is called the {\em Mackey toplogy} on $E$ associated with $\tau$, and if $\mu=\tau$, the space $E$ is called a {\em Mackey space}.

An analogous notion in the class of locally quasi-convex (lqc for short) abelian groups was introduced in \cite{CMPT}. For an abelian topological group $(G,\tau)$ we denote by $\widehat{G}$ the group of all continuous characters of $(G,\tau)$ (for all relevant definitions see the next section). Two topologies  $\mu$ and $\nu$ on an abelian group $G$  are said to be {\em compatible } if $\widehat{(G,\mu)}=\widehat{(G,\nu)}$. Following \cite{CMPT},  an lqc  abelian group $(G,\mu)$ is called a {\em Mackey group} if for every lqc group topology $\nu$ on $G$ compatible with $\tau$  it follows that $\nu\leq\mu$.

Not every Mackey lcs is a Mackey group. In \cite{Gab-Cp} we show that the space $C_p(X)$, which is a Mackey space for every Tychonoff space $X$, is a Mackey group if and only if it is barrelled. In particular, this result
shows  that there are even metrizable lcs which are not Mackey groups that gives a negative answer to a question posed in \cite{DMPT}.  Only very recently,  answering Question 4.4 of \cite{Gab-Mackey},  Au\ss enhofer \cite{Aus3} and the author \cite{Gabr-A(s)-Mackey} independently have shown that there are lqc groups which do not admit a Mackey group topology. For historical remarks, references and open questions we referee the reader to \cite{Gab-Mackey,MarPei-Tar}. In  Question 4.3 of \cite{Gab-Mackey}, we ask: For which Tychonoff spaces $X$ the free  lcs $L(X)$ is a Mackey space or a Mackey group?  Below we give a complete answer to this question.

\begin{theorem} \label{t:Mackey-space-L(X)}
For a  Tychonoff space $X$, the following assertions are equivalent:
\begin{enumerate}
\item[{\rm (i)}] $L(X)$ is a Mackey group;
\item[{\rm (ii)}] $L(X)$ is a Mackey space;
\item[{\rm (iii)}] $X$ is discrete.
\end{enumerate}
\end{theorem}
In particular, Theorem \ref{t:Mackey-space-L(X)} essentially strengthen Theorem 6.4 of \cite{GM} which states that $L(X)$ is quasibarrelled if and only of $X$ is discrete.


\section{Proof of Theorem \ref{t:Mackey-space-L(X)}}


We start from some necessary definitions and notations. Let $X$ be a Tychonoff  space. The space $X$ is called a {\em $k_\IR$-space} if every real-valued function on $X$ which is continuous on every compact subset of $X$ is continuous on $X$.
A subset $A$ of $X$ is called {\em functionally bounded in $X$} if every continuous real-valued function on $X$ is bounded on $A$, and $X$ is a {\em $\mu$-space} if every functionally bounded subset of $X$ has compact closure. The Dieudonn\'{e} completion $\mu X$ of $X$ is always a $\mu$-space.

Denote by $\mathbb{S}$ the unit circle group and set $\Ss_+ :=\{z\in  \Ss:\ {\rm Re}(z)\geq 0\}$.
Let $G$ be an abelian topological group.   A character $\chi\in \widehat{G}$  is a continuous homomorphism from $G$ into $\mathbb{S}$.
A subset $A$ of $G$ is called {\em quasi-convex} if for every $g\in G\setminus A$ there exists   $\chi\in \widehat{G}$ such that $\chi(g)\notin \Ss_+$ and $\chi(A)\subseteq \Ss_+$.
The group $G$ is called {\em locally quasi-convex} if it admits a neighborhood base at the neutral element $0$ consisting of quasi-convex sets.
Every real locally convex space is a locally quasi-convex group by Proposition 2.4 of \cite{Ban}.

Following \cite{Mar},  the {\em  free locally convex space}  $L(X)$ on a Tychonoff space $X$ is a pair consisting of a locally convex space $L(X)$ and  a continuous map $i: X\to L(X)$  such that every  continuous map $f$ from $X$ to a locally convex space  $E$ gives rise to a unique continuous linear operator ${\bar f}: L(X) \to E$  with $f={\bar f} \circ i$. The free locally convex space $L(X)$ always exists and is essentially unique. The set $X$ forms a Hamel basis for $L(X)$ and  the map $i$ is a topological embedding, see \cite{Rai,Usp2}.

Let $X$ be a Tychonoff space. For  $\chi = a_1 x_1+\cdots +a_n x_n\in L(X)$ with distinct $x_1,\dots, x_n\in X$ and  nonzero $a_1,\dots,a_n\in\IR$, we set
\[
\| \chi\|:=|a_1|+\cdots +|a_n|, \; \mbox{ and } \; \mathrm{supp}(\chi):=\{ x_1,\dots, x_n\}.
\]


For an lcs $E$, we denote by $E'$ the topological dual space of $E$. 
For a cardinal number $\kappa$, the classical Banach space $c_0(\kappa)$ consists of all bounded functions $g: \kappa \to \IR$ such that the set $\{ i\in \kappa: |g(i)|\geq \e\}$ is finite for every $\e>0$ and is endowed with the supremum norm $\| \cdot\|_\infty$.

We denote by $\CC(X)$  the space $C(X)$ of all real-valued continuous functions on $X$ endowed with the compact-open topology $\tau_k$.
The support of a function $f\in C(X)$ is denoted by $\supp(f)$.
Denote by $M_c(X)$ the space of all real regular Borel measures on $X$ with compact support. It is well-known that the dual space of $\CC(X)$ is $M_c(X)$, see \cite[Proposition~7.6.4]{Jar}. For every $x\in X$, we denote by $\delta_x \in M_c(X)$ the evaluation map (Dirac measure), i.e. $\delta_x(f):=f(x)$ for every $f\in C(X)$. Denote by $\tau_e$ the polar topology on $M_c(X)$ defined by the family of all equicontinuous  pointwise bounded subsets of $C(X)$.
We shall use the following deep result of Uspenski\u{\i} \cite{Usp2}.

\begin{theorem}[\cite{Usp2}] \label{t:Free-complete-L}
Let $X$ be a Tychonoff space and let $\mu X$ be the Dieudonn\'{e} completion of $X$. Then the completion $\overline{L(X)}$ of $L(X)$ is topologically isomorphic to $\big(M_c(\mu X),\tau_e\big)$.
\end{theorem}

We need also the following corollary of Theorem \ref{t:Free-complete-L} noticed in \cite{Gab-Respected}.
\begin{corollary}[\cite{Gab-Respected}] \label{p:Ck-Mc-compatible}
Let $X$ be a $\mu$-space. Then the topology $\tau_e$ on $M_c(X)$ is compatible with the duality $(\CC(X),M_c(X))$.
\end{corollary}

\begin{proof}
It is well-known that $L(X)'=C(X)$, see \cite{Rai}. Now Theorem \ref{t:Free-complete-L} implies $(M_c(X),\tau_e)'=L(X)'=C(X)$.
\end{proof}

The next lemma follows from Proposition 2.5 of \cite{Gab-Mackey}, we give its proof for the sake of completeness of the paper.
\begin{lemma} \label{l:Mackey-space-group}
If a real lcs $(E,\tau)$ is a Mackey group, then it is a Mackey space.
\end{lemma}

\begin{proof}
Let $\nu$ be a locally convex vector topology on $E$ compatible with $\tau$. Applying Proposition 2.3 of \cite{Ban} we obtain $\widehat{(E,\nu)}=\widehat{(E,\tau)}$. Hence $\nu$ is a locally quasi-convex group topology (see Proposition 2.4 of \cite{Ban}) compatible with $\tau$. Therefore $\nu\leq\tau$ since $(E,\tau)$ is a Mackey group. Thus $(E,\tau)$ is a Mackey space.
\end{proof}

We need the following characterization of non-discrete Tychonoff spaces.
\begin{proposition} \label{p:Mackey-space-L(X)-0}
A Tychonoff space $X$ is not  discrete if and only if there exist an infinite cardinal $\kappa$, a point $z\in X$, a family $\{ g_i\}_{i\in\kappa}$ of continuous functions from $X$ to $[0,2]$ and a family $\{ U_i\}_{i\in\kappa}$ of open subsets of $X$ such that
\begin{enumerate}
\item[{\rm (i)}] $\supp(g_i) \subseteq U_i $ for every $i\in\kappa$;
\item[{\rm (ii)}] $U_i\cap U_j=\emptyset $ for all distinct $i,j\in\kappa$;
\item[{\rm (iii)}] $z\not\in U_i$ for every $i\in\kappa$ and $z\in \cl\big(\bigcup_{i\in\kappa} \{ x\in X: g_i(x)\geq 1\}\big)$.
\end{enumerate}
\end{proposition}

\begin{proof}
The sufficiency follows from (i)-(iii) which cannot hold simultaneously for discrete spaces. To prove the necessity we consider two cases.

\smallskip
{\em Case 1. There is a continuous function $h: X\to [0,1]$ such that the set $L:=\{ x\in X: h(x)>0\}$ is not closed.}  So there is a $z\in \cl(L)$ such that $h(z)=0$.
We distinguish between two subcases.

\smallskip
{\em Subcase 1.1. For every neighborhood $W$ of $z$, the closure $\overline{h(W)}$ of $h(W)$ contains an interval of the form $[0,\e)$ for some $\e>0$.} For every $n\in\NN$, set
\[
t_n(x) := \max\left\{ h(x)-\frac{1}{3n+1},0\right\} \cdot \max\left\{ 0, \frac{1}{3n-1}-h(x)\right\}
\]
and
\[
A_n := \sup\{ t_n(x): x\in X\}.
\]
Let $m$ be the least natural number such that $A_n>0$ for every $n>m$. For every $n>m$, set
\[
U_n := h^{-1} \left( \frac{1}{3n+1},  \frac{1}{3n-1}\right)\;\; \mbox{ and } \;\; g_n(x):= \frac{2}{A_n} \cdot t_n(x),
\]
Then, for every $n>m$, we have $g_n(X)\subseteq [0,2]$, $\supp(g_n) \subseteq U_n$ and $z\not\in U_n$. Clearly, (i) and (ii) are fulfilled and $z\not\in U_n$ for every $n>m$. So to check  that the sequences $\{ U_n: n>m\}$ and  $\{ g_n: n>m\}$   and the point $z$ satisfy (i)-(iii) we have to show that $z\in \cl\big( \bigcup_{n>m} g_n^{-1} \big( [1,2]\big)\big)$.

Fix arbitrarily a neighborhood $W$ of $z$ in $X$. Then, by assumption, $\overline{h(W)}$ contains $[0,\e)$ for some $\e\in(0,1)$. Therefore, if $n_0 >(1+3\e)/(3\e)$ there is a $y\in W$ such that $t_{n_0}(y)\geq (1/2) A_{n_0}$, and hence $g_{n_0}(y)\geq 1$. Thus $g_{n_0}^{-1} \big( [1,2]\big) \cap W$ is not empty and hence $z\in \cl\big( \bigcup_{n>m} g_n^{-1} \big( [1,2]\big)\big)$.

\smallskip
{\em Subcase 1.2. There is a neighborhood $W$ of $z$ such that the closure $\overline{h(W)}$ of $h(W)$ does not contain an interval of the form $[0,\e)$}. Then there exist sequences $\{a_n\}_{n\in\NN}$ and $\{b_n\}_{n\in\NN}$ in $(0,1)$ converging to zero such that
\[
b_{n+1} < a_n < b_n, \; [b_{n+1},a_n]\cap \overline{h(W)}=\emptyset \mbox{ and } (a_n,b_n)\cap \overline{h(W)}\not=\emptyset, \quad \forall n\in\NN.
\]
Set $a_0:=1$ and
\[
c_n:=\frac{1}{2} (b_{n+1} +a_n) \; \mbox{ and } \; d_n:=\frac{1}{2} (b_{n} +a_{n-1}), \quad \forall n\in\NN.
\]
Then $c_n < a_n <b_n < d_n < 1$. For every $n\in\NN$, let $r_n(x)$ be the piecewise linear continuous function from $[0,1]$ to $[0,1]$ such that
\[
r_n\big( [0,c_n]\cup[d_n,1]\big) = \{ 0\} \mbox{ and } r_n\big( [a_n,b_n]\big) = \{ 1\},
\]
and set
\[
U_n:= h^{-1}(c_n,d_n) \; \mbox{ and } \; g_n(x) := r_n \big( h(x)\big),  \quad x\in X.
\]
By construction, the sequences $\{ U_n: n\in\NN\}$ and  $\{ g_n: n\in\NN\}$   and the point $z$ satisfy (i) and (ii) and $z\not\in U_n$ for every $n\in\NN$. Let us show that every neighborhood $U$ of $z$ contains elements of $\bigcup_{n\in\NN} g_n^{-1} \big( \{1\}\big)$. We can assume that $U\subseteq W$. Since $[b_{n+1},a_n]\cap \overline{h(W)}=\emptyset$ we obtain that $h(W) \subseteq \{ 0\}\cup \bigcup_{n\in\NN} (a_n,b_n)$. Therefore, if $y\in U$ and $n_0\in\NN$ is such that $h(y)\in (a_{n_0},b_{n_0})$ (such a $y$ exists because $z\in \cl(L)$), then $g_{n_0}(y)=1$.

\smallskip
{\em Case 2. For every continuous function $h: X\to [0,1]$ the set $\{ x\in X: h(x)>0\}$ is closed.}

\smallskip
We claim that $X$ has a neighborhood base containing closed-and-open sets. Indeed, since $X$ is Tychonoff, for every point $x\in X$ and each open neighborhood $U$ of $x$ there is a continuous function $h:X\to [0,1]$ such that $h(x)=1$ and $h(X\setminus U)=\{ 0\}$. It remains to note that, by assumption,  the open neighborhood $h^{-1}\big( (0,1]\big)\subseteq U$ of $x$ also is closed.

\smallskip
Now, by the assumption of the proposition,  there is a non-isolated point $z\in X$. By the Zorn lemma, there exists a maximal (under inclusion) family $\U=\{ U_i: i\in\kappa\}$ of pairwise disjoint closed-and-open sets such that $z\not\in U_i$ for every $i\in \kappa$. The maximality of $\U$ and the claim imply that $z\in \cl\big( \bigcup \U \big)$. For every $i\in\kappa$, let $g_i$ be the characteristic function of $U_i$. Clearly, the families $\U$ and $\{ g_i: i\in\kappa\}$ and the point $z$ satisfy conditions (i)-(iii).
\end{proof}

The following proposition is crucial for the proof of Theorem \ref{t:Mackey-space-L(X)}.

\begin{proposition} \label{p:Mackey-space-L(X)-1}
Let $X$ be  a Dieudonn\'{e} complete space. If $\big(M_c(X),\tau_e\big)$ is a Mackey space, then $X$ is discrete.
\end{proposition}

\begin{proof}
Suppose for a contradiction that $X$ is not discrete. Then, by Proposition \ref{p:Mackey-space-L(X)-0}, there exist an infinite cardinal $\kappa$, a point $z\in X$, a family $\{ g_i\}_{i\in\kappa}$ of continuous functions from $X$ to $[0,2]$ and a family $\{ U_i\}_{i\in\kappa}$ of open subsets of $X$ satisfying (i)-(iii) of that proposition. Define a map $R:M_c(X)\to c_0(\kappa)$ by
\[
R(\mu):=\big( \mu(g_i)\big), \quad \mu\in M_c(X).
\]

\smallskip
{\em Claim 1. The map $R$ is well-defined}.

\smallskip
Indeed, let $\mu\in M_c(X)$ be a positive measure. Since $\mu$ is finite and $\sigma$-additive, the condition (ii) of Proposition \ref{p:Mackey-space-L(X)-0}  implies that for every $\e>0$ the number of indices $i\in\kappa$ for which $\mu(U_i)\geq \e$ is finite. Now the claim follows from the inclusion $\supp(g_i) \subseteq U_i $ (see (i)) and the inequalities $0\leq \mu(g_i) \leq 2\mu(U_i)$.

\smallskip
Consider a map  $T: M_c(X)\to \big( M_c(X),\tau_e\big)\times c_0(\kappa)$ defined by
\[
 T(\mu):=\big( \mu, R(\mu)\big), \;\; \forall \mu\in M_c(X).
\]
The map $T$ is well-defined by Claim 1. Denote by $\TTT$ the locally convex vector topology on $M_c(X)$ induced from the product $\big( M_c(X),\tau_e\big)\times c_0(\kappa)$.

\smallskip
{\em Claim 2. The topology $\TTT$ is compatible with $\tau_e$.}

\smallskip
First we note that for every $(\lambda_i)\in \big(c_0(\kappa)\big)'=\ell_1(\kappa)$, the function $\sum_i \lambda_i g_i$ belongs to $C(X)$. The Hahn--Banach extension theorem implies that every $\chi\in (M_c(X),\TTT)'$ has the form
\[
\chi=\big( F, (\lambda_i)\big), \mbox{ where } F\in (M_c(X),\tau_e)' \mbox{ and } (\lambda_i) \in \ell_1(\kappa).
\]
By  Corollary \ref{p:Ck-Mc-compatible}, we have $F\in C(X)$ and hence $G:=F+\sum_{i\in\kappa} \lambda_i f_i\in C(X)$. Therefore
\[
\chi(\mu)=\mu(F)+\sum_{i\in\kappa} \lambda_i \cdot \mu(g_i)=\mu\left( F+\sum_{i\in\kappa} \lambda_i g_i\right)=\mu(G), \; \forall \mu\in M_c(X).
\]
Applying Corollary \ref{p:Ck-Mc-compatible} once again we obtain $\chi=G\in (M_c(X),\tau_e)'$ as desired.

\smallskip
{\em Claim 3. We claim that $\tau_e <\TTT$.}

\smallskip
Indeed, it is clear that $\tau_e\leq \TTT$. Set
\[
S:=\{\delta_{x}: \mbox{ there is an } i\in\kappa \mbox{ such that } g_i(x)\geq 1 \} \subseteq M_c(X).
\]
To show that $\tau_e \not= \TTT$, we shall prove that
(1) $
\delta_z \in \cl_{\tau_e}(S)$,   and (2) $ \delta_z \not\in \cl_{\TTT}(S).
$

To prove that $\delta_z \in \cl_{\tau_e}(S)$, fix arbitrarily a standard neighborhood
\[
[K;\e]:=\{ \mu \in M_c(X): |\mu(f)|<\e \; \forall f\in K\}
\]
of zero in $(M_c(X),\tau_e)$, where $K$ is a pointwise bounded equicontinuous subset of $C(X)$ and $\e>0$. Choose a neighborhood $U$ of $z$ such that
\[
|f(x) - f(z)|<\e, \quad \forall f\in K, \quad \forall x\in U.
\]
By (iii) of Proposition \ref{p:Mackey-space-L(X)-0}, take an $i_0\in\kappa$ and $x_{i_0}\in U$ such that $g_{i_0}(x_{i_0})\geq 1$. Then $\delta_{x_{i_0}}\in S$ and
\[
\big| (\delta_{x_{i_0}} -\delta_z)(f)\big| = |f(x_{i_0}) - f(z)|<\e, \quad \forall f\in K.
\]
Thus $\delta_{x_{i_0}} \in \delta_z + [K;\e]$ and hence $\delta_z \in \cl_{\tau_e}(S)$.

To show that $\delta_z \not\in \cl_{\TTT}(S)$, consider the neighborhood $W:= M_c(X)\times U$ of zero in $\TTT$, where $U=\{ g\in c_0(\kappa): \| g\|_\infty \leq 1/2\}$. Fix arbitrarily $\delta_x\in S$ and choose $j\in\kappa$ such that $g_j(x)\geq 1$. Then the $j$th coordinate $\delta_{x}(g_j)$ of $R(\delta_{x})$ satisfies the following (in  the last equality we use (i) and (ii) of Proposition \ref{p:Mackey-space-L(X)-0})
\[
\big|  \delta_{x}(g_j) - \delta_z(g_j) \big| = | g_j(x) -g_j(z)| = g_j(x)\geq 1 >1/2.
\]
Therefore $R(\delta_{x})- R(\delta_z) \not\in U$ and hence $ \delta_{x}-\delta_z\not\in W$. As $x$ was arbitrary we obtain $\delta_z \not\in \cl_{\TTT}(S)$.

\smallskip
Finally, Claims 2 and 3 imply that $\big(M_c(X),\tau_e\big)$ is not a Mackey space. This contradiction shows that $X$ must be discrete.
\end{proof}

Theorem  \ref{t:Mackey-space-L(X)} follows from the next more general result.
\begin{theorem} \label{t:Mackey-space-L(X)-M(X)}
For a Tychonoff space $X$ the following assertions are equivalent:
\begin{enumerate}
\item[{\rm (i)}] $L(X)$ is a Mackey group;
\item[{\rm (ii)}] $L(X)$ is a Mackey space;
\item[{\rm (iii)}] $\big(M_c(\mu X),\tau_e\big)$ is a Mackey group;
\item[{\rm (iv)}] $\big(M_c(\mu X),\tau_e\big)$ is a Mackey space;
\item[{\rm (v)}] $X$ is discrete.
\end{enumerate}
\end{theorem}

\begin{proof}
(i)$\Rightarrow$(ii) and (iii)$\Rightarrow$(iv) follow from Lemma \ref{l:Mackey-space-group}.

(ii)$\Rightarrow$(iv) It is well known that the completion of a Mackey space is a Mackey space, see Proposition 8.5.8 of \cite{Jar}. Therefore, by Theorem \ref{t:Free-complete-L},  the space $\big(M_c(\mu X),\tau_e\big)$ is a Mackey space. 

(iv)$\Rightarrow$(v) By Proposition \ref{p:Mackey-space-L(X)-1}, $\mu X$ is discrete. Thus $X$ is discrete as well.

(v)$\Rightarrow$(i),(iii) Since $X$ is discrete Theorem \ref{t:Free-complete-L} implies $L(X)=\big(M_c(X),\tau_e\big)$. Therefore $L(X)$  is a barrelled space by \cite[Theorem~6.4]{GM}, and hence $L(X)$ is a Mackey group by  \cite[Theorem~4.2]{CMPT}.
\end{proof}


\bibliographystyle{amsplain}

\end{document}